\documentclass[a4paper]{article}
\usepackage{geometry}
\usepackage{amscd,amssymb,amsthm,amsmath,epsfig,boxedminipage}
\usepackage{amssymb}
\usepackage{latexsym}
\usepackage{amsmath}
\usepackage{indentfirst}
\usepackage{graphicx}
\usepackage[colorlinks=true]{hyperref}
\usepackage{mathrsfs} 
\usepackage{chngcntr}
\usepackage{color}

\newtheorem{Theorem}{Theorem}[section]

\newtheorem{Lemma}{Lemma}[section]

\newtheorem{Remark}{Remark}[section]
\newtheorem{Definition}{Definition}[section]

\begin{document}

\title{\bf On a class of logarithmic Schr\"odinger equations via a new perturbation method}
\author{
	{Chen Huang}\thanks{Email: chenhuangmath111@163.com}\\
	\small School of Mathematics, University of Shanghai for Science and Technology, Shanghai, 200093, P.R. China\\
	\small Yunnan Key Laboratory of Modern Analytical Mathematics and Applications,  Kunming, 650500, P. R. China.\\
	{Zhipeng Yang}\thanks{Corresponding author: yangzhipeng326@163.com}\\
	\small Department\, of \,Mathematics, Yunnan\, Normal\, University, Kunming, 650500, P. R. China.\\
	\small Yunnan Key Laboratory of Modern Analytical Mathematics and Applications, Kunming, 650500, P. R. China.\\
}

\date{}
\maketitle

\begin{abstract}
\noindent In this paper, we consider the following logarithmic Schr\"odinger equation
	\[
	-\Delta u + V(x)u = u \log u^{2},\quad x\in\mathbb{R}^{N}.
	\]
	Assuming that \(V\in C(\mathbb{R}^{N},\mathbb R)\), \(V\) is bounded away from zero, and \(V(x)\to+\infty\) as \(|x|\to\infty\), we develop a new perturbative variational approach to overcome the lack of \(C^{1}\)-smoothness of the associated functional and prove the existence and multiplicity of nontrivial weak solutions. 
\end{abstract}

\ \ \ \ {\bf Keywords:} Logarithmic Schr\"{o}dinger equations; Perturbation methods; Minimax methods.
\par
\ \ \ \ {\bf 2010 AMS Subject Classification:} 35J20; 35B20; 49J35.

\section{Introduction}
The logarithmic Schr\"{o}dinger equation is defined by
\[
i \frac{\partial \Psi}{\partial t} = -\Delta\Psi + (V(x) + E)\Psi - \Psi \log \Psi^2, \quad x \in \mathbb{R}^N, \tag{NLSE}
\]
where \(N \geq 3\), and \(-\Delta\) denotes the Laplacian operator. Here, \( V \) represents a continuous potential function, a model that has received comparatively limited attention in the literature. A fundamental objective involves the investigation of standing wave solutions to (NLSE) of the form \(\Psi(x,t) = \exp(-iEt)u(x)\), where \( u \) satisfies the elliptic equation:
\begin{equation}\label{eq1.1}
  -\Delta u + V(x)u = u \log u^2, \quad x \in \mathbb{R}^N.  
\end{equation}

The logarithmic Schr\"{o}dinger equation exhibits broad applicability across diverse fields including quantum mechanics, quantum optics, nuclear physics, transport and diffusion phenomena, open quantum systems, effective quantum gravity, superfluid theory, and Bose-Einstein condensation. By using the variational method, the existence, multiplicity, and concentration behavior of solutions to equation \eqref{eq1.1} have been extensively studied; see references \cite{MR4048330,MR4097474,MR719365,MR3195154,MR2570526,MR3451965,MR3951962,MR3385171,MR3894545,MR4066104}. 

The natural candidate for the energy functional associated with equation \eqref{eq1.1} takes the form:
\[
I(u) = \frac{1}{2} \int_{\mathbb{R}^N} \left( |\nabla u|^2 + (V(x) + 1)u^2 \right) \text{d}x - \frac{1}{2} \int_{\mathbb{R}^N} u^2 \log u^2  \text{d}x. 
\]
It is readily verified that critical points of \(I\) correspond to solutions of \eqref{eq1.1}. However, this functional is not well-defined on the space \( H^1(\mathbb{R}^N) \), since there exist functions \( u \in H^1(\mathbb{R}^N) \) for which the integral \( \int_{\mathbb{R}^N} u^2 \log u^2 \text{d}x \) diverges to \( -\infty \). This is exemplified by the function:
\[
u(x) = 
\begin{cases} 
(|x|^{N/2} \log |x|)^{-1}, & |x| \geq 3, \\
0, & |x| \leq 2.
\end{cases}
\]

To address this technical difficulty, various approaches have been developed in the literature. In \cite{MR719365}, Cazenave formulated the problem within an appropriately constructed Banach space \( W \), endowed with a Luxembourg-type norm that ensures the functional \( I: W \to \mathbb{R} \) is well-defined and of class \( C^1 \). In \cite{MR2570526}, Guerrero et al. established the existence of nontrivial solutions for the three-dimensional logarithmic Schr\"{o}dinger equation by penalizing the logarithmic nonlinearity near the origin. In \cite{MR3195154}, for constant potentials \(V(x)\), d'Avenia et al. applied nonsmooth critical point theory to a lower semicontinuous functional to prove the existence of infinitely many radial solutions to equation \eqref{eq1.1}. More recently, Squassina and Szulkin \cite{MR3385171} adopted an alternative strategy by working directly in the unconstrained space \( H^1(\mathbb{R}^N) \). They decomposed the functional \( I \) into the sum of a \( C^1 \) functional and a convex lower semicontinuous functional, thereby establishing the existence of infinitely many geometrically distinct solutions for a class of logarithmic Schr\"{o}dinger equations with periodic potentials. Recently, in \cite{MR3951962}, Shuai employed a constrained variational method to prove the existence of positive and sign-changing solutions for equation \eqref{eq1.1} under various potential conditions.

In a significant contribution \cite{MR3894545}, Wang and Zhang obtained a profound connection between the power-law nonlinearity equation and the logarithmic Schr\"{o}dinger equation. Inspired by this work, Zhang and Wang \cite{MR4066104} proposed a global perturbation of the variational functional with power nonlinearity
\[
I_{q}(u) = \frac{1}{2} \int_{\mathbb{R}^N} \left( |\nabla u|^2 + V(x)u^2 \right) \text{d}x - \int_{\mathbb{R}^N}\left(\frac{2}{q(q-2)}|u|^q-\frac{1}{q-2}u^2\right)  \text{d}x. 
\]
As a result, a family of $C^2$ functionals $I_{q}(u)$ in $H^1(\mathbb{R}^N)$ were considered. The critical points of $I_{q}$ with uniform bounded functional values converge to
some critical points of $I$ as $q\to2^{+}$.

In contrast to all the variational methods mentioned above, this paper presents a novel perturbation approach for the logarithmic Schr\"{o}dinger equation. Indeed, when we consider equation \eqref{eq1.1} by using the classical critical point theory we encounter the difficulties caused by the lack of an appropriate working space. For smoothness one would need to work in a space smaller than $H^{1}(\mathbb{R}^N)$ to control the logarithmic term, for example, we may choose $L^{p}(\mathbb{R}^N)\cap H^{1}(\mathbb{R}^N)$ with $p\in(\max\{\frac{2N-4}{N+2},1\},2)$ as the working space for \eqref{eq1.1}. Due to the following algebraic inequality, for \(t\in\mathbb{R}\),
\[
|t^{2}\log t^2|\leq C_{\varepsilon}(|t|^{2-\varepsilon}+|t|^{2+\varepsilon}),~\text{with}~\varepsilon\in(0,2-p),
\]
where $C_\varepsilon$ is a positive constant depending on $\varepsilon$. Combining the Sobolev embedding theorem, the functional corresponding to the logarithmic term is known to exhibit smoothness in $L^{p}(\mathbb{R}^N)\cap H^{1}(\mathbb{R}^N)$. But it seems difficult to obtain bounded Palais--Smale sequences in this smaller working space. In order to overcome the difficulties, when considering the logarithmic Schr\"{o}dinger equation in the space $L^{p}(\mathbb{R}^N)\cap H^{1}(\mathbb{R}^N)$, we introduce a $L^p$-perturbation to equation \eqref{eq1.1}. We then consider a perturbed functional
\begin{equation*}
\begin{aligned}
I_{\lambda}(u)&=\frac{\lambda}{p}\int_{\mathbb{R}^N}|u|^p\text{d}x+\frac{1}{2}\int_{\mathbb{R}^N}\left(|\nabla u|^{2}+(V(x)+1)u^{2}\right)\text{d}x-\frac{1}{2}\int_{\mathbb{R}^N}u^{2}\log u^{2}\text{d}x.
\end{aligned}
\end{equation*}
The idea is to obtain existence of critical points of $I_{\lambda}$ for $\lambda>0$ small and to establish suitable estimates for the critical points as $\lambda\to0^{+}$ so that we may pass to the limit to get solutions of the original problem.

By using our new perturbation method, we study the logarithmic Schr\"odinger equation \eqref{eq1.1} with \(V\) satisfying:
\begin{itemize}
	\item[(V)] \(V\in C(\mathbb{R}^{N},\mathbb{R})\), \(\displaystyle \inf_{x\in\mathbb{R}^{N}} V(x)\ge V_{0}>0\) and \(V(x)\to +\infty\) as \(|x|\to\infty\).
\end{itemize}

\begin{Remark}\label{Rem1.1}
	If \(v\) is a weak solution of 
	\[
	-\Delta v+\big(V(x)-\log \lambda^{2}\big)v \;=\; v\,\log v^{2}
	\quad x\in\mathbb{R}^{N},
	\]
	for some \(\lambda\neq 0\), then \(u:=\lambda v\) is a weak solution of \eqref{eq1.1}. 
	Indeed,
	\[
	-\Delta(\lambda v)+V(x)\,\lambda v 
	= \lambda\big(-\Delta v+V(x)v\big)
	= \lambda\,v\big(\log v^{2}+\log \lambda^{2}\big)
	= \lambda v \log\!\big(\lambda^{2}v^{2}\big).
	\]
	In particular, if \(V\) is bounded from below, choosing \(|\lambda|\) small enough so that \(\log \lambda^{2}<\inf\limits_{x\in\mathbb{R}^{N}}V(x)\) yields 
	\(V(x)-\log \lambda^{2}\ge V_{0}'>0\) for all \(x\). 
	Thus, by this scaling we may, without loss of generality, assume \(\inf\limits_{x\in\mathbb{R}^{N}}V(x)>0\).
\end{Remark}

Our main results are as follows.

\begin{Theorem}\label{thm1.1}
	Let \(N\ge 3\). Assume (V) holds. Then problem \eqref{eq1.1} admits a nontrivial weak solution $u_0$.
\end{Theorem}

\begin{Theorem}\label{thm1.2}
	Let \(N\ge 3\). Assume (V) holds. Then problem \eqref{eq1.1} admits infinitely many weak solutions $\{u_{j}\}_{j=1}^{\infty}$ with $I(u_j)\to+\infty$ as $j\to\infty$.
\end{Theorem}

\begin{Remark}\label{rem:perturbation}
In contrast to previous strategies for the logarithmic Schr\"odinger equation---such as working in a specially tailored Orlicz-type space to restore smoothness, penalizing the logarithmic term near the origin, applying nonsmooth critical point theory to lower semicontinuous functionals, splitting $I$ into a $C^{1}$ part plus a convex lower semicontinuous\ part, or continuing from power nonlinearities $|u|^{q}$ as $q\to2^{+}$---the present paper introduces a $L^p$-perturbation at the level of the equation:
\[
I_{\lambda}(u)=\frac{\lambda}{p}\!\int_{\mathbb{R}^{N}}\!|u|^{p}\text{d}x+\frac{1}{2}\!\int_{\mathbb{R}^{N}}\!\big(|\nabla u|^{2}+(V(x)+1)u^{2}\big)\text{d}x-\frac{1}{2}\!\int_{\mathbb{R}^{N}}u^{2}\log u^{2}\text{d}x,
\]
and works in the intersection space $X=L^{p}(\mathbb{R}^{N})\cap H^{1}_{V}(\mathbb{R}^{N})$ with $p\in(\max\{\frac{2N-4}{N+2},1\},2)$. This has two decisive effects:\\ (i) it restores $C^{1}$-smoothness and compactness (Palais--Smale) for $I_{\lambda}$ without changing the original logarithmic structure;\\
(ii) it yields mountain--pass critical points with bounds independent of $\lambda$, so that one can pass to the limit $\lambda\to0^{+}$ and recover a nontrivial weak solution of the original problem in \(H_V^1\). 
\end{Remark} 

To the best of our knowledge, this is the first use of a small $L^p$-perturbation to regularize the variational framework for the logarithmic Schr\"odinger equation while staying in the natural Hilbert space $H^{1}_{V}$ in the limit. The approach is robust and may be adapted to other logarithmic-type or borderline nonlinearities where direct $C^{1}$-smoothness fails.

The paper is organized as follows. Section 2 introduces the perturbation framework and proves Theorem \ref{thm1.1}. Section 3 proves Theorem \ref{thm1.2}.

\noindent{\bf Notations.} In this paper we use the following notations:
\begin{itemize}
    \item $|u|_{s}=\left(\int_{\mathbb{R}^N}|u|^{s}\text{d}x\right)^{1/s}$ denotes the usual norm in $L^{s}$-space.
    \item $o_n(1)$ denotes a vanishing sequence.
    \item $u_{+}=\max\{u,0\}$ and $u_{-}=-\min\{u,0\}$.
    \item The weak and strong convergence are denoted by $u_n\rightharpoonup u$ and $u_n\to u$, respectively, as $n\to\infty$.
    \item $C, C_1, C_2,\cdots$ denote different positive constants.
\end{itemize}

\section{Proof of Theorem \ref{thm1.1}}
Let
\[
H^{1}_{V}(\mathbb{R}^{N})
:=\Big\{u\in H^{1}(\mathbb{R}^{N}) \;:\; \int_{\mathbb{R}^{N}} V(x)\,u^{2}\,\text{d}x<\infty \Big\},
\]
endowed with the inner product
\[
\langle u,v\rangle_{V}
:=\int_{\mathbb{R}^{N}}\big(\nabla u\cdot\nabla v+V(x)\,u\,v\big)\,\text{d}x
\]
and its norm $\|u\|_{H^{1}_{V}}:=\langle u,u\rangle_{V}^{1/2}$.
We call \(u\in H^{1}_{V}(\mathbb{R}^{N})\) a weak solution of \eqref{eq1.1} if 
\(\displaystyle \int_{\mathbb{R}^{N}} u^{2}|\log u^{2}|\,\text{d}x<\infty\) and
\[
\int_{\mathbb{R}^{N}}\big(\nabla u\cdot\nabla \varphi+V(x)\,u\,\varphi\big)\,\text{d}x
=\int_{\mathbb{R}^{N}} u\,\log u^{2}\,\varphi\,\text{d}x,
\quad \text{for all }\,~\varphi\in C_{0}^{\infty}(\mathbb{R}^{N}).
\]
The equation \eqref{eq1.1} is formally the Euler--Lagrange equation of the functional
\begin{equation}\label{eq2.1}
	I(u)
	=\frac{1}{2}\int_{\mathbb{R}^{N}}\big(|\nabla u|^{2}+(V(x)+1)\,u^{2}\big)\,\text{d}x
	-\frac{1}{2}\int_{\mathbb{R}^{N}} u^{2}\log u^{2}\,\text{d}x.
\end{equation}
We call \(u\in H^1_V(\mathbb{R}^N)\) a critical point of \(I\) (i.e. \(I'(u)=0\)), if \(\int_{\mathbb{R}^N}u^2|\log u^2|\,\text{d}x<\infty\) and 
\[
\int_{\mathbb{R}^{N}}\big(\nabla u\cdot\nabla \psi+V(x)\,u\,\psi\big)\,\text{d}x
-\int_{\mathbb{R}^{N}} u\,\log u^{2}\,\psi\,\text{d}x=0,
\quad \text{for all }\,~\psi\in H^{1}_{V}(\mathbb{R}^{N})~\text{and}~\displaystyle \int_{\mathbb{R}^{N}} \psi^{2}|\log \psi^{2}|\,\text{d}x<\infty.
\]
We note \(\int_{\mathbb{R}^{N}}| u\,\log u^{2}\,\psi|\,\text{d}x<\infty\) by \cite[Lemma~2.1]{MR3613541}.

By the logarithmic Sobolev inequality as in \cite{MR1817225} and \(\inf\limits_{x\in\mathbb{R}^{N}}V(x)>0\), we have \(I(u)>-\infty\) for all \(u\in H^{1}_{V}(\mathbb{R}^{N})\).
However, there exists \(u\in H^{1}_{V}(\mathbb{R}^{N})\) with \(\int_{\mathbb{R}^{N}}u^{2}\log u^{2}\,\text{d}x=-\infty\), so in general \(I\) takes values in \((-\infty,+\infty]\) and is not \(C^{1}\) on \(H^{1}_{V}(\mathbb{R}^{N})\).

To overcome this lack of smoothness, we consider a new perturbative variational approach. Fix \(\lambda\in(0,1]\) and \(p\in(\max\{\frac{2N-4}{N+2},1\},2)\), we define the working space $X$ and its norm
\[
X:=L^{p}(\mathbb{R}^{N})\cap H^{1}_{V}(\mathbb{R}^{N}),
\]
\[
\|u\|_{X} := |u|_{p} + \|u\|_{H^{1}_{V}},
\quad\text{for all }~u\in X.
\]
We look for \(u\in X\) such that
\begin{equation}\label{eq2.2}
	\begin{aligned}
	&\lambda\!\int_{\mathbb{R}^{N}}\! |u|^{p-2}u\,\phi\,\text{d}x
	+\!\int_{\mathbb{R}^{N}}\!\big(\nabla u\cdot\nabla \phi+V(x)\,u\,\phi\big)\,\text{d}x\\
	&=\!\int_{\mathbb{R}^{N}}\! u\,\log u^{2}\,\phi\,\text{d}x,
	\quad\text{for all }~\phi\in X.
\end{aligned}
\end{equation}
The associated \(C^{1}\) functional on \(X\) is defined by 
\begin{equation*}
	\begin{aligned}
		I_{\lambda}(u)
		&=\frac{\lambda}{p}\int_{\mathbb{R}^{N}}|u|^{p}\,\text{d}x
		+\frac{1}{2}\int_{\mathbb{R}^{N}}\big(|\nabla u|^{2}+(V(x)+1)\,u^{2}\big)\,\text{d}x
		-\frac{1}{2}\int_{\mathbb{R}^{N}} u^{2}\log u^{2}\,\text{d}x.
	\end{aligned}
\end{equation*}
Using the elementary estimate
\[
|t^{2}\log t^{2}|
\;\le\; C_{\varepsilon}\,\big(|t|^{2-\varepsilon}+|t|^{2+\varepsilon}\big),
\quad\text{for all }~t\in\mathbb{R},
\]
where $\varepsilon\in(0,1)$ is small enough and $C_{\varepsilon}$ is a positive constant depending on $\varepsilon$.
Together with the embeddings \(X\hookrightarrow L^{2\pm\varepsilon}(\mathbb{R}^{N})\), this implies that \(I_{\lambda}\in C^{1}(X)\). Moreover, for any $\phi\in X$, its derivative is
\[
\begin{aligned}
\big\langle I'_{\lambda}(u),\phi\big\rangle=\lambda\!\int_{\mathbb{R}^{N}}\! |u|^{p-2}u\,\phi\,\text{d}x
+\!\int_{\mathbb{R}^{N}}\!\big(\nabla u\cdot\nabla \phi+V(x)\,u\,\phi\big)\,\text{d}x
-\!\int_{\mathbb{R}^{N}}\! u\,\log u^{2}\,\phi\,\text{d}x,
\end{aligned}
\]
so that \(I'_{\lambda}(u)=0\) if and only if \(u\) satisfies \eqref{eq2.2}.

\begin{Lemma}\cite{MR1615056}\label{lem2.1}
	Define the linear operator
	\[
	\mathcal{H}:H^{1}_{V}(\mathbb{R}^{N})\to \big(H^{1}_{V}(\mathbb{R}^{N})\big)^{*},\qquad
	\langle \mathcal{H}u,v\rangle_{V}
	=\int_{\mathbb{R}^{N}}\big(\nabla u\cdot\nabla v+V(x)\,u\,v\big)\,\text{d}x,\quad\text{for all }~v\in H^{1}_{V}(\mathbb{R}^N),
	\]
	which is bounded and strongly monotone:
	\[
	\langle \mathcal{H}u-\mathcal{H}v,\,u-v\rangle_{V}
	=\|u-v\|_{H^{1}_{V}}^{2}\quad\text{for all }~u,v\in H^{1}_{V}(\mathbb{R}^{N}).
	\]
	 Let $\langle\cdot,\cdot\rangle$ be the duality pairing between a Banach space $X$ and its dual $X^{*}$. For $m\in(1,2)$, define
	\[
	\mathcal{L}_{m}:L^{m}(\mathbb{R}^{N})\to \big(L^{m}(\mathbb{R}^{N})\big)^{*},\qquad
	\langle \mathcal{L}_{m}u,v\rangle
	=\int_{\mathbb{R}^{N}}|u|^{\,m-2}u\,v\,\text{d}x,\quad\text{for all }~v\in L^{m}(\mathbb{R}^N).
	\]
	The operator $\mathcal{L}_{m}$ is bounded, semicontinuous, coercive, and monotone; in particular it is of type (S): if $u_n\rightharpoonup u$ in $L^{m}(\mathbb{R}^{N})$ and $\langle \mathcal{L}_{m}u_n,\,u_n-u\rangle\to0$, then $u_n\to u$ in $L^{m}(\mathbb{R}^{N})$.
\end{Lemma}

First, we show $I_{\lambda}$ satisfies the Palais--Smale condition. 

\begin{Lemma}\label{lem2.2}
	Assume (V) holds. For any fixed $\lambda\in(0,1]$, the functional $I_{\lambda}$ satisfies the Palais--Smale condition on $X$.
\end{Lemma}

\begin{proof}
	Let $\{u_n\}\subset X$ be a Palais--Smale sequence, that is,
\[
I_{\lambda}(u_n)\ \text{is bounded}
\quad\text{and}\quad
I'_{\lambda}(u_n)\to0
\quad\text{in }X^{*}.
\]
Since $I_{\lambda}(u_n)$ is bounded and $I'_{\lambda}(u_n)\to0$ in $X^*$, we compute
	\[
	\begin{aligned}
		C+o_n(1)\,\|u_n\|_{X}
		\geq 2I_{\lambda}(u_n)-\big\langle I'_{\lambda}(u_n),u_n\big\rangle=\frac{(2-p)\lambda}{p}\!\int_{\mathbb{R}^N}\!|u_n|^{p}\,\text{d}x
		+\int_{\mathbb{R}^N}\!u_n^{2}\,\text{d}x.
	\end{aligned}
	\]
	Hence there exists $C>0$ such that
	\begin{equation}\label{eq2.4}
		C+o_n(1)\,\|u_n\|_{X}\ \ge\ C\,|u_n|_{p}^{p}+|u_n|_{2}^{2}.
	\end{equation}
	
	On the other hand, by the logarithmic Sobolev inequality (see \cite{MR1817225}),
	for any $a>0$,
	\[
	\int_{\mathbb{R}^N}u_n^{2}\log u_n^{2}\,\text{d}x
	\;\le\; \frac{a^{2}}{\pi}\,|\nabla u_n|_{2}^{2}
	+\Big(\log |u_n|_{2}^{2}-N(1+\log a)\Big)\,|u_n|_{2}^{2}.
	\]
	Choose $a>0$ small so that $\frac{a^{2}}{\pi}\le \frac12$, then
	\begin{equation}\label{eq2.5}
		\int_{\mathbb{R}^N}u_n^{2}\log u_n^{2}\,\text{d}x
		\;\le\; \frac12\,|\nabla u_n|_{2}^{2}
		+C\,\big(1+|u_n|_{2}^{2}\log |u_n|_{2}^{2}\big).
	\end{equation}
	From \eqref{eq2.4} we have, for $n$ large,
	\[
	|u_n|_{2}^{2}\ \le\ C+o_n(1)\,\|u_n\|_{X}
\le C\,(1+\|u_n\|_{X}).
	\]
	Fix $\delta\in(0,p-1)$. Using the elementary bound $\log r\le C_{\delta}\,(1+r^{\,\delta})$ for all $r\ge 0$,
	we estimate
	\[
	|u_n|_{2}^{2}\log |u_n|_{2}^{2}
	\ \le\ C_{\delta}\,\big(1+(1+\|u_n\|_{X})^{\delta}\big)\,(1+\|u_n\|_{X})
	\ \le\ C'_{\delta}\,\big(1+\|u_n\|_{X}^{\,1+\delta}\big).
	\]
	Plugging this into \eqref{eq2.5} yields
	\begin{equation}\label{eq:log-est}
		\int_{\mathbb{R}^N}u_n^{2}\log u_n^{2}\,\text{d}x
		\ \le\ \frac12\,|\nabla u_n|_{2}^{2}
		+C\,\big(1+\|u_n\|_{X}^{\,1+\delta}\big).
	\end{equation}
	Therefore, using \eqref{eq2.4} and \eqref{eq:log-est} in the identity for $2I_{\lambda}(u_n)$, we have
	\[
	\begin{aligned}
		C\ \ge\ 2I_{\lambda}(u_n)
		&=\frac{2\lambda}{p}\!\int_{\mathbb{R}^N}\!|u_n|^{p}\,\text{d}x
		+\int_{\mathbb{R}^N}\!\big(|\nabla u_n|^{2}+(V(x)+1)u_n^{2}\big)\,\text{d}x
		-\int_{\mathbb{R}^N}\!u_n^{2}\log u_n^{2}\,\text{d}x \\
		&\ge \frac{2\lambda}{p}\,|u_n|_{p}^{p}
		+\frac12\,\|u_n\|_{H^{1}_{V}}^{2}
		- C\,\big(1+\|u_n\|_{X}^{\,1+\delta}\big),\quad\text{for}~n~\text{large}.
	\end{aligned}
	\]
	Since $1+\delta<p$ (recall $\delta\in(0,p-1)$), the right-hand side controls $\|u_n\|_{X}$ for large values,
	hence $\{u_n\}$ is a bounded sequence in $X$.
	
	Up to a subsequence, $u_n\rightharpoonup u$ in $X$. By the compact embedding induced by (V) (see e.g.\ \cite{MR1349229}),
\[
u_n\to u\quad\text{in }L^s(\mathbb R^N),\qquad 2\le s<2^*.
\]
In particular,
\[
u_n\to u\quad\text{in }L^2(\mathbb R^N).
\]
Since $\{u_n\}$ is bounded in $L^p(\mathbb R^N)$, interpolation between $L^p(\mathbb R^N)$ and $L^2(\mathbb R^N)$ gives
\[
u_n\to u\quad\text{in }L^r(\mathbb R^N),\qquad p<r<2.
\]
Taking $r=\frac{p+2}{2}$, we obtain
\[
u_n\to u\quad\text{in }L^{\frac{p+2}{2}}(\mathbb R^N).
\]
	Since $I'_{\lambda}(u_n)\to0$ in $X^{*}$ and $\{u_n\}$ is a bounded sequence in $X$, we have
	\[
	\big\langle I'_{\lambda}(u_n),u_n-u\big\rangle=o_{n}(1).
	\]
	Expanding the derivative gives
	\begin{equation}\label{eq:split}
		\begin{aligned}
		&\lambda\!\int_{\mathbb{R}^N}\!|u_n|^{p-2}u_n(u_n-u)\,\text{d}x+\!\int_{\mathbb{R}^N}\!\Big(\nabla u_n\cdot\nabla (u_n-u)+V(x)u_n(u_n-u)\Big)\,\text{d}x\\
		&=\int_{\mathbb{R}^N}\!u_n\log u_n^{2}\,(u_n-u)\,\text{d}x+o_{n}(1).
	\end{aligned}
	\end{equation}
	Due to \(p\in(\max\{\frac{2N-4}{N+2},1\},2)\), we can choose $\varepsilon>0$ small so that 
	\[
	(1+\varepsilon)\,\frac{p+2}{p}\;<\;2^{*}~\text{and}~(1-\varepsilon)\,\frac{p+2}{p}\;>\;p.
	\]
	Using $|t\log t^{2}|\le C_{\varepsilon}\big(|t|^{1-\varepsilon}+|t|^{1+\varepsilon}\big)$ and H\"older's inequality,
	\begin{equation}\label{eq8}
	\begin{aligned}
		\left|\int_{\mathbb{R}^N}u_n\log u_n^{2}\,(u_n-u)\,\text{d}x\right|
		&\le C_{\varepsilon}\,\Big|\,|u_n|^{1-\varepsilon}+|u_n|^{1+\varepsilon}\Big|_{\frac{p+2}{p}}
		\,|u_n-u|_{\frac{p+2}{2}}
		\;=\;o_{n}(1),
	\end{aligned}
	\end{equation}
	since $\{u_n\}$ is bounded in $L^{(1\pm\varepsilon)\frac{p+2}{p}}(\mathbb{R}^N)$ and $u_n\to u$ in $L^{\frac{p+2}{2}}(\mathbb{R}^N)$.
	
	It follows from \eqref{eq:split}, \eqref{eq8} and Lemma \ref{lem2.1} that
	\begin{equation}\label{eq9}
	\lambda\,\langle \mathcal{L}_{p}u_n,\,u_n-u\rangle
	+\langle \mathcal{H}u_n,\,u_n-u\rangle_{V}
	=o_{n}(1).
	\end{equation}
	Since $\{u_n\}$ is bounded in $X=L^{p}(\mathbb{R}^{N})\cap H^{1}_{V}(\mathbb{R}^{N})$, up to a subsequence
	\[
	u_n \rightharpoonup u \ \text{ in }\ H^{1}_{V}(\mathbb{R}^{N})
	\quad\text{and}\quad
	u_n \rightharpoonup u \ \text{ in }\ L^{p}(\mathbb{R}^{N}).
	\]
	Using the linearity of $\mathcal{H}$ and the definition of the $H^{1}_{V}$-inner product,
	\begin{equation}\label{eq10}
	\langle \mathcal{H}u_n,\,u_n-u\rangle_{V}
	=\langle \mathcal{H}(u_n-u),\,u_n-u\rangle_{V}
	+\langle \mathcal{H}u,\,u_n-u\rangle_{V}
	=\|u_n-u\|_{H^{1}_{V}}^{2}+o_{n}(1),
	\end{equation}
	because $\langle \mathcal{H}u,\,\cdot\,\rangle_{V}$ is a continuous functional on $H^{1}_{V}(\mathbb{R}^{N})$ and $u_n-u\rightharpoonup 0$ in $H^{1}_{V}(\mathbb{R}^{N})$.
	Moreover, by monotonicity of $\mathcal{L}_{p}$,
	\begin{equation}\label{eq11}
	\langle \mathcal{L}_{p}u_n,\,u_n-u\rangle \ge \langle \mathcal{L}_{p}u,\,u_n-u\rangle=o_{n}(1),
	\end{equation}
	since $u_n-u\rightharpoonup 0$ in $L^{p}(\mathbb{R}^{N})$ and $\mathcal{L}_{p}u\in (L^{p}(\mathbb{R}^{N}))^{*}$.
	
	Combining \eqref{eq9} and \eqref{eq10}, it gives
	\[
	\lambda\,\langle \mathcal{L}_{p}u_n,\,u_n-u\rangle+\|u_n-u\|_{H^{1}_{V}}^{2}=o_{n}(1).
	\]
	Hence $\limsup\limits_{n\to\infty}\,\langle \mathcal{L}_{p}u_n,\,u_n-u\rangle\le 0$, while the monotonicity bound \eqref{eq11} yields $\liminf\limits_{n\to\infty}\,\langle \mathcal{L}_{p}u_n,\,u_n-u\rangle\ge 0$; thus
	\[
	\langle \mathcal{L}_{p}u_n,\,u_n-u\rangle\to 0
	\quad\text{and}\quad
	\|u_n-u\|_{H^{1}_{V}}\to 0.
	\]
	Finally, since $\mathcal{L}_{p}$ is of type (S), from $\langle \mathcal{L}_{p}u_n,\,u_n-u\rangle\to 0$ and $u_n\rightharpoonup u$ in $L^{p}(\mathbb{R}^{N})$ we conclude $u_n\to u$ in $L^{p}(\mathbb{R}^{N})$. Therefore $u_n\to u$ in $X$.
\end{proof}

\begin{Lemma}\label{lem:mp-geometry}
Assume (V) holds, let
\[
p\in\left(\max\left\{\frac{2N-4}{N+2},1\right\},2\right)
\]
and \(\lambda\in(0,1]\). Then there exist constants \(\rho>0\) and
\(\alpha>0\), independent of \(\lambda\), such that
\begin{equation}\label{eq:mp-lower}
\inf_{\substack{u\in X\\ \|u\|_{H^{1}_{V}}=\rho}}
I_{\lambda}(u)
\ge \alpha>0 .
\end{equation}
Moreover, there exists \(e_0\in X\), independent of \(\lambda\), such that
\[
\|e_0\|_{H_V^1}>\rho,
\qquad
I_\lambda(e_0)<0
\quad\text{for all }\lambda\in(0,1].
\]
\end{Lemma}

\begin{proof}
Fix \(\delta\in(0,2^*-2)\) and \(u\in X\). Using that
\(t^2\log t^2<0\) on \(t\in(0,1)\) and, for \(t\ge1\),
\[
t^2\log t^2\le C_\delta t^{2+\delta},
\]
we get
\[
\begin{aligned}
I_{\lambda}(u)
&=
\frac{\lambda}{p}\int_{\mathbb R^N}|u|^p\,\text{d}x
+\frac12\int_{\mathbb R^N}\bigl(|\nabla u|^2+(V(x)+1)u^2\bigr)\,\text{d}x
-\frac12\int_{\mathbb R^N}u^2\log u^2\,\text{d}x  \\
&\ge
\frac12\|u\|_{H_V^1}^2
-\frac12\int_{\{x\in\mathbb R^N:\ |u(x)|\ge1\}}u^2\log u^2\,\text{d}x  \\
&\ge
\frac12\|u\|_{H_V^1}^2
-C_\delta\int_{\mathbb R^N}|u|^{2+\delta}\,\text{d}x .
\end{aligned}
\]
By the Sobolev embedding, we have
\[
|u|_{2+\delta}\le C\|u\|_{H_V^1}.
\]
Therefore
\[
I_{\lambda}(u)
\ge
\frac12\|u\|_{H_V^1}^2
-C\|u\|_{H_V^1}^{2+\delta}.
\]
Consequently, there exist \(\rho,\alpha>0\), independent of
\(\lambda\in(0,1]\), such that
\[
\inf_{\substack{u\in X\\ \|u\|_{H_V^1}=\rho}}
I_{\lambda}(u)
\ge \alpha>0 .
\]

Now choose a nonzero function \(e\in C_0^\infty(\mathbb R^N)\). Along the ray
\(t\mapsto te\), we have
\[
\begin{aligned}
I_1(te)
&=
\frac{t^p}{p}\int_{\mathbb R^N}|e|^p\,\text{d}x
+\frac{t^2}{2}\int_{\mathbb R^N}\bigl(|\nabla e|^2+(V(x)+1)e^2\bigr)\,\text{d}x  \\
&\quad
-\frac{t^2\log t^2}{2}\int_{\mathbb R^N}e^2\,\text{d}x
-\frac{t^2}{2}\int_{\mathbb R^N}e^2\log e^2\,\text{d}x .
\end{aligned}
\]
Since \(p<2\) and \(e\not\equiv0\), the negative term
\[
-t^2\log t\int_{\mathbb R^N}e^2\,\text{d}x
\]
dominates as \(t\to+\infty\). Hence
\[
I_1(te)\to-\infty
\quad\text{as }t\to+\infty.
\]
Thus we may choose \(t_0>0\) so large that
\[
I_1(t_0e)<0,
\qquad
\|t_0e\|_{H_V^1}>\rho .
\]
Set
\[
e_0=t_0e .
\]
Since \(I_\lambda(u)\le I_1(u)\) for every \(u\in X\) and every
\(\lambda\in(0,1]\), we obtain
\[
I_\lambda(e_0)\le I_1(e_0)<0
\quad\text{for all }\lambda\in(0,1].
\]
The proof is complete.
\end{proof}

Let \(e_0\in X\) be given by Lemma~\ref{lem:mp-geometry}. Set
\[
\gamma_0(s)=s e_0,\qquad s\in[0,1].
\]
Then
\[
\gamma_0(0)=0,\qquad \gamma_0(1)=e_0,
\]
and
\[
I_\lambda(\gamma_0(1))<0
\quad\text{for all }\lambda\in(0,1].
\]
Define the mountain-pass value
\[
c(\lambda)=
\inf_{\gamma\in\Gamma_\lambda}
\sup_{s\in[0,1]}I_{\lambda}(\gamma(s)),
\]
where
\[
\Gamma_\lambda=
\left\{
\gamma\in C([0,1],X):\ \gamma(0)=0,\ \gamma(1)=e_0
\right\}.
\]
By \eqref{eq:mp-lower}, we have
\[
c(\lambda)\ge \alpha>0.
\]
Since \(I_{\lambda}\) satisfies the Palais--Smale condition by Lemma~\ref{lem2.2},
the Mountain Pass Theorem yields \(u_{\lambda}\in X\) such that
\begin{equation}\label{9}
I_{\lambda}'(u_{\lambda})=0,
\qquad
I_{\lambda}(u_{\lambda})=c(\lambda)\ge\alpha>0 .
\end{equation}

We shall use the following uniform estimate.
\begin{Lemma}\label{lem2.3}
There exist two positive constants \(m_1,m_2>0\), independent of
\(\lambda\in(0,1]\), such that
\[
m_1\le I_{\lambda}(u_{\lambda})\le m_2,
\]
where \(u_{\lambda}\in X\) is the mountain-pass critical point of
\(I_{\lambda}\) obtained in \eqref{9}.
\end{Lemma}

\begin{proof}
By Lemma~\ref{lem:mp-geometry} and the Palais--Smale condition
given by Lemma~\ref{lem2.2}, for every \(\lambda\in(0,1]\) there exists
\(u_{\lambda}\in X\) such that
\[
I_{\lambda}'(u_{\lambda})=0,
\qquad
I_{\lambda}(u_{\lambda})=c(\lambda)\ge \alpha>0,
\]
where \(\alpha\) is independent of \(\lambda\). Thus we choose
\[
m_1=\alpha .
\]

For the upper bound, let \(e_0\in X\) be the endpoint constructed in
Lemma~\ref{lem:mp-geometry}, and set
\[
\gamma_0(s)=s e_0,\qquad s\in[0,1].
\]
Then
\[
\gamma_0(0)=0,\qquad \gamma_0(1)=e_0,
\]
and hence \(\gamma_0\in\Gamma_\lambda\) for every \(\lambda\in(0,1]\).
Since \(\lambda\le1\), we have
\[
I_\lambda(u)\le I_1(u),
\qquad u\in X.
\]
Therefore, by the definition of \(c(\lambda)\),
\[
\begin{aligned}
c(\lambda)
&=
\inf_{\gamma\in\Gamma_\lambda}\sup_{s\in[0,1]}I_\lambda(\gamma(s))  \\
&\le
\sup_{s\in[0,1]}I_\lambda(\gamma_0(s))  \\
&\le
\sup_{s\in[0,1]}I_1(\gamma_0(s))
=:m_2 .
\end{aligned}
\]
The map \(s\mapsto I_1(\gamma_0(s))\) is continuous on the compact interval
\([0,1]\), and hence
\[
m_2<+\infty .
\]
Moreover, \(m_2\) is independent of \(\lambda\). Thus
\[
I_\lambda(u_\lambda)=c(\lambda)\le m_2 .
\]
Together with the lower bound, this proves the claim.
\end{proof}

Inspired by the variational perturbation approach introduced in \cite{MR3259554,MR2988727,MR2983045} for quasi-linear Schr\"{o}dinger equations, we obtain the following result:

\begin{Lemma}\label{lem2.5}
Suppose that \(\lambda_n\to0^+\),
\[
I_{\lambda_n}'(u_n)=0,
\qquad
I_{\lambda_n}(u_n)\le C,
\]
where \(C>0\) is independent of \(n\). Then, up to a subsequence, there exists
\(u_0\in H_V^1(\mathbb R^N)\) such that
\[
u_n\rightharpoonup u_0
\quad\text{in }H_V^1(\mathbb R^N),
\]
and \(u_0\) is a weak solution of \eqref{eq1.1}. Moreover, if
\[
\liminf_{n\to\infty}I_{\lambda_n}(u_n)>0,
\]
then \(u_0\ne0\).
\end{Lemma}

\begin{proof}
First we prove the uniform bounds. Since \(I'_{\lambda_n}(u_n)=0\) and
\(I_{\lambda_n}(u_n)\le C\), we have
\[
\begin{aligned}
2C
\ge
2I_{\lambda_n}(u_n)-\langle I_{\lambda_n}'(u_n),u_n\rangle=
\frac{2-p}{p}\lambda_n\int_{\mathbb R^N}|u_n|^p\,\text{d}x
+\int_{\mathbb R^N}u_n^2\,\text{d}x .
\end{aligned}
\]
Hence
\begin{equation}\label{eq:step1-basic}
\lambda_n|u_n|_p^p\le C,
\qquad
|u_n|_2^2\le C .
\end{equation}
By the logarithmic Sobolev inequality, choosing \(a>0\) sufficiently small, we obtain
\[
\int_{\mathbb R^N}u_n^2\log u_n^2\,\text{d}x
\le
\frac12|\nabla u_n|_2^2
+C\bigl(1+|u_n|_2^2|\log |u_n|_2^2|\bigr).
\]
Using \eqref{eq:step1-basic}, this gives
\[
\int_{\mathbb R^N}u_n^2\log u_n^2\,\text{d}x
\le
\frac12|\nabla u_n|_2^2+C .
\]
Therefore,
\[
\begin{aligned}
2C
&\ge
2I_{\lambda_n}(u_n) \\
&=
\frac{2\lambda_n}{p}\int_{\mathbb R^N}|u_n|^p\,\text{d}x
+
\int_{\mathbb R^N}\bigl(|\nabla u_n|^2+(V(x)+1)u_n^2\bigr)\,\text{d}x
-
\int_{\mathbb R^N}u_n^2\log u_n^2\,\text{d}x  \\
&\ge
\frac{2\lambda_n}{p}|u_n|_p^p
+
\frac12\|u_n\|_{H_V^1}^2
-C .
\end{aligned}
\]
Thus
\begin{equation}\label{eq:uniform-HV-bound}
\|u_n\|_{H_V^1}\le C,
\qquad
\lambda_n|u_n|_p^p\le C .
\end{equation}
Consequently, up to a subsequence,
\[
u_n\rightharpoonup u_0
\quad\text{in }H_V^1(\mathbb R^N)
\]
for some \(u_0\in H_V^1(\mathbb R^N)\). Since \(V(x)\to+\infty\) as
\(|x|\to\infty\), the embedding
\[
H_V^1(\mathbb R^N)\hookrightarrow L^s(\mathbb R^N)
\]
is compact for every \(s\in[2,2^*)\). Hence
\[
u_n\to u_0
\quad\text{in }L^s(\mathbb R^N),\quad 2\le s<2^*,
\]
and
\[
u_n(x)\to u_0(x)
\quad\text{a.e. in }\mathbb R^N .
\]

We next pass to the limit in the equation. For every
\(\varphi\in C_0^\infty(\mathbb R^N)\), we have
\[
\lambda_n\int_{\mathbb R^N}|u_n|^{p-2}u_n\varphi\,\text{d}x
+
\int_{\mathbb R^N}\bigl(\nabla u_n\cdot\nabla\varphi+V(x)u_n\varphi\bigr)\,\text{d}x
=
\int_{\mathbb R^N}u_n\log u_n^2\,\varphi\,\text{d}x .
\]
By H\"older's inequality and \eqref{eq:uniform-HV-bound},
\[
\begin{aligned}
\left|
\lambda_n\int_{\mathbb R^N}|u_n|^{p-2}u_n\varphi\,\text{d}x
\right|
&\le
\lambda_n |u_n|_p^{p-1}|\varphi|_p  \\
&=
\lambda_n^{\frac1p}
\bigl(\lambda_n |u_n|_p^p\bigr)^{\frac{p-1}{p}}
|\varphi|_p
\to0 .
\end{aligned}
\]
Moreover,
\[
\int_{\mathbb R^N}\bigl(\nabla u_n\cdot\nabla\varphi+V(x)u_n\varphi\bigr)\,\text{d}x
\to
\int_{\mathbb R^N}\bigl(\nabla u_0\cdot\nabla\varphi+V(x)u_0\varphi\bigr)\,\text{d}x .
\]
To treat the logarithmic term, fix \(q\in(2,2^*)\) and define
\(f(t)=t\log t^2\) for \(t\ne0\), \(f(0)=0\). Then \(f\) is continuous and
\[
|f(t)|\le C_q\bigl(1+|t|^{q-1}\bigr),
\qquad t\in\mathbb R.
\]
Since \(u_n\to u_0\) in \(L^q_{\mathrm{loc}}(\mathbb R^N)\), the continuity
of the Nemytskii operator associated with \(f\) yields
\[
u_n\log u_n^2\to u_0\log u_0^2
\quad\text{in }L^{\frac{q}{q-1}}_{\mathrm{loc}}(\mathbb R^N).
\]
Therefore, for every \(\varphi\in C_0^\infty(\mathbb R^N)\),
\[
\int_{\mathbb R^N}u_n\log u_n^2\,\varphi\,\text{d}x
\to
\int_{\mathbb R^N}u_0\log u_0^2\,\varphi\,\text{d}x .
\]
Letting \(n\to\infty\), we obtain
\begin{equation}\label{weaksolution}
\int_{\mathbb R^N}\bigl(\nabla u_0\cdot\nabla\varphi+V(x)u_0\varphi\bigr)\,\text{d}x
=
\int_{\mathbb R^N}u_0\log u_0^2\,\varphi\,\text{d}x
\end{equation}
for all \(\varphi\in C_0^\infty(\mathbb R^N)\).

It remains to prove the global logarithmic integrability required in the
definition of weak solutions. Set
\[
F(t)=(t^2\log t^2)_+,
\qquad
G(t)=(t^2\log t^2)_- .
\]
For any \(q\in(2,2^*)\), there exists \(C_q>0\) such that
\[
F(t)\le C_q |t|^q
\quad\text{for all }t\in\mathbb R.
\]
Since \(u_0\in H_V^1(\mathbb R^N)\subset L^q(\mathbb R^N)\), we have
\begin{equation}\label{eq:positive-part-u0}
\int_{\mathbb R^N}F(u_0)\,\text{d}x<+\infty .
\end{equation}

Let \(\eta_R\in C_0^\infty(\mathbb R^N)\) satisfy
\[
0\le \eta_R\le1,\qquad
\eta_R=1\ \text{in }B_R,\qquad
\eta_R=0\ \text{in }\mathbb R^N\setminus B_{2R},
\qquad
|\nabla\eta_R|\le \frac{C}{R}.
\]
We may also assume that \(\eta_R(x)\to1\) for a.e. \(x\in\mathbb R^N\) as
\(R\to+\infty\). Since
\[
u_0\log u_0^2\in L_{\mathrm{loc}}^{\frac{q}{q-1}}(\mathbb R^N)
\quad\text{for every }q\in(2,2^*),
\]
and
\[
\eta_Ru_0\in H_0^1(B_{2R})\subset L^q(B_{2R})
\quad\text{for every }q\in[2,2^*),
\]
we may use \(\eta_Ru_0\) as a test function in \eqref{weaksolution} by a
standard density argument. Thus
\[
\int_{\mathbb R^N}\eta_R\bigl(|\nabla u_0|^2+V(x)u_0^2\bigr)\,\text{d}x
+
\int_{\mathbb R^N}u_0\nabla u_0\cdot\nabla\eta_R\,\text{d}x
=
\int_{\mathbb R^N}\eta_Ru_0^2\log u_0^2\,\text{d}x .
\]
Since
\[
u_0^2\log u_0^2=F(u_0)-G(u_0),
\]
we get
\[
\begin{aligned}
\int_{\mathbb R^N}\eta_R G(u_0)\,\text{d}x
&=
\int_{\mathbb R^N}\eta_R F(u_0)\,\text{d}x
-
\int_{\mathbb R^N}\eta_R\bigl(|\nabla u_0|^2+V(x)u_0^2\bigr)\,\text{d}x  \\
&\quad
-
\int_{\mathbb R^N}u_0\nabla u_0\cdot\nabla\eta_R\,\text{d}x  \\
&\le
\int_{\mathbb R^N}F(u_0)\,\text{d}x
+
\left|
\int_{\mathbb R^N}u_0\nabla u_0\cdot\nabla\eta_R\,\text{d}x
\right|.
\end{aligned}
\]
Since \(V(x)\ge V_0>0\), we have \(u_0\in L^2(\mathbb R^N)\). Hence
\[
\left|
\int_{\mathbb R^N}u_0\nabla u_0\cdot\nabla\eta_R\,\text{d}x
\right|
\le
\frac{C}{R}|u_0|_2|\nabla u_0|_2
\to0 .
\]
Therefore,
\[
\sup_{R>1}
\int_{\mathbb R^N}\eta_R G(u_0)\,\text{d}x<+\infty .
\]
Letting \(R\to+\infty\) and using Fatou's lemma, we obtain
\[
\int_{\mathbb R^N}G(u_0)\,\text{d}x<+\infty .
\]
Combining this with \eqref{eq:positive-part-u0}, we conclude that
\[
\int_{\mathbb R^N}u_0^2|\log u_0^2|\,\text{d}x<+\infty .
\]
Thus \(u_0\) is a weak solution of \eqref{eq1.1}.

Finally, assume in addition that
\[
\liminf_{n\to\infty}I_{\lambda_n}(u_n)>0.
\]
We prove that \(u_0\ne0\). Suppose by contradiction that \(u_0=0\). Then
\(u_n\to0\) in \(L^q(\mathbb R^N)\) for every \(q\in(2,2^*)\). Testing
\(I_{\lambda_n}'(u_n)=0\) with \(u_n\), we have
\[
\lambda_n\int_{\mathbb R^N}|u_n|^p\,\text{d}x
+
\int_{\mathbb R^N}\bigl(|\nabla u_n|^2+V(x)u_n^2\bigr)\,\text{d}x
=
\int_{\mathbb R^N}u_n^2\log u_n^2\,\text{d}x .
\]
Since
\[
(t^2\log t^2)_+\le C_q |t|^q,
\quad t\in\mathbb R,
\]
we get
\[
\lambda_n\int_{\mathbb R^N}|u_n|^p\,\text{d}x
+
\int_{\mathbb R^N}\bigl(|\nabla u_n|^2+V(x)u_n^2\bigr)\,\text{d}x
\le
C_q\int_{\mathbb R^N}|u_n|^q\,\text{d}x
\to0 .
\]
Consequently,
\[
\lambda_n\int_{\mathbb R^N}|u_n|^p\,\text{d}x\to0,
\qquad
\int_{\mathbb R^N}u_n^2\,\text{d}x\to0.
\]
Using
\[
2I_{\lambda_n}(u_n)-\langle I_{\lambda_n}'(u_n),u_n\rangle
=
\frac{2-p}{p}\lambda_n\int_{\mathbb R^N}|u_n|^p\,\text{d}x
+
\int_{\mathbb R^N}u_n^2\,\text{d}x,
\]
we obtain
\[
I_{\lambda_n}(u_n)\to0,
\]
which contradicts
\[
\liminf_{n\to\infty}I_{\lambda_n}(u_n)>0.
\]
Therefore \(u_0\ne0\).
\end{proof}

\noindent\textbf{Proof of Theorem \ref{thm1.1}.}
By Lemma \ref{lem2.3}, for every \(\lambda\in(0,1]\) there exists a mountain-pass critical point
\(u_{\lambda}\in X\) of \(I_{\lambda}\), and there are constants \(m_1,m_2>0\),
independent of \(\lambda\), such that
\[
m_1\le I_{\lambda}(u_{\lambda})\le m_2 .
\]
Let \(\lambda_n\to0^+\) and set \(u_n=u_{\lambda_n}\). Then
\[
I_{\lambda_n}'(u_n)=0,
\qquad
I_{\lambda_n}(u_n)\le m_2 .
\]
By Lemma \ref{lem2.5}, passing to a subsequence if necessary, there exists
\(u_0\in H_V^1(\mathbb R^N)\) such that
\[
u_n\rightharpoonup u_0
\quad\text{in }H_V^1(\mathbb R^N),
\]
and \(u_0\) is a weak solution of \eqref{eq1.1}. Moreover, since
\[
\liminf_{n\to\infty}I_{\lambda_n}(u_n)\ge m_1>0,
\]
Lemma \ref{lem2.5} further implies that \(u_0\ne0\). Hence \(u_0\) is a nontrivial
weak solution of \eqref{eq1.1}. This completes the proof.
\hfill\(\square\)

\section{Proof of Theorem \ref{thm1.2}}
In this section, we prove the existence of an unbounded sequence of positive critical values for the energy functional:
\begin{align*}
I_{\lambda}(u)
&=\frac{\lambda}{p}\int_{\mathbb{R}^N}\!|u|^{p}\,\text{d}x
	+\frac12\int_{\mathbb{R}^N}\!\big(|\nabla u|^{2}+(V(x)+1)u^{2}\big)\,\text{d}x
	-\frac12\int_{\mathbb{R}^N}\!u^{2}\log u^{2}\,\text{d}x.
\end{align*}
We begin by recalling the definition of the Krasnosel'skii genus, a topological index associated with the $\mathbb{Z}_2$ symmetry group. 
\begin{Definition}
    Let $\mathcal{E}$ denote the family of all sets $A \subset X \setminus \{0\}$ that are closed in $X$ and symmetric with respect to the origin. For any nonempty $A \in \mathcal{E}$, the genus $\gamma(A)$ is defined as the smallest integer $n \in \mathbb{N}$ such that there exists a continuous odd mapping $\phi: A \to \mathbb{R}^n \setminus \{0\}$. We set $\gamma(A) = \infty$ if no such finite $n$ exists, and $\gamma(\emptyset) = 0$.
\end{Definition}
 The concept of genus, together with the deformation lemma below (see \cite[Theorem A.4]{MR1400007}), plays a fundamental role in both the construction and estimation of the minimax values of the functional $I_\lambda$.
 
\begin{Lemma}\label{lem3.1}
Assume that \(J\in C^1(X,\mathbb{R})\) satisfies the Palais--Smale condition. Set
\[
K_c=\{u\in X:\ J(u)=c\ \text{and}\ J'(u)=0\}.
\]
Given \(c\in\mathbb{R}\), \(\overline{\theta}>0\), and a neighborhood
\(\mathcal{O}\) of \(K_c\), there exist \(\theta\in(0,\overline{\theta})\) and
\(\sigma\in C([0,1]\times X,X)\) such that
\begin{enumerate}
\item[\((1)\)] \(\sigma(0,u)=u\) for all \(u\in X\);

\item[\((2)\)] if \(J(u)\notin[c-\overline{\theta},c+\overline{\theta}]\), then
\(\sigma(t,u)=u\) for all \(t\in[0,1]\);

\item[\((3)\)] if \(u\notin\mathcal{O}\) and \(J(u)\le c+\theta\), then
\[
J(\sigma(1,u))\le c-\theta;
\]

\item[\((4)\)] if \(K_c=\emptyset\) and \(J(u)\le c+\theta\), then
\[
J(\sigma(1,u))\le c-\theta;
\]

\item[\((5)\)] if \(J\) is even, then \(\sigma(t,\cdot)\) is odd for every
\(t\in[0,1]\), namely
\[
\sigma(t,-u)=-\sigma(t,u),\qquad u\in X.
\]
\end{enumerate}
\end{Lemma}

The following lemma is a further consequence of Lemma \ref{lem2.5}, which is necessary for the proof of multiplicity.
\begin{Lemma}\label{lem3.2}
Assume that \((V)\) holds and
\[
p\in\left(\max\left\{1,\frac{2N-4}{N+2}\right\},2\right).
\]
Suppose that \(\lambda_n\to0^+\),
\[
I_{\lambda_n}'(u_n)=0,
\qquad
I_{\lambda_n}(u_n)\le C,
\]
where \(C>0\) is independent of \(n\). Then, up to a subsequence, there exists
\(u_0\in H_V^1(\mathbb R^N)\) such that
\[
u_n\to u_0
\quad\text{in }H_V^1(\mathbb R^N).
\]
Moreover, \(u_0\) is a weak solution of \eqref{eq1.1} and
\[
\int_{\mathbb R^N}\bigl(|\nabla u_0|^2+V(x)u_0^2\bigr)\,\text{d}x
=
\int_{\mathbb R^N}u_0^2\log u_0^2\,\text{d}x .
\]
Furthermore,
\[
\lambda_n\int_{\mathbb R^N}|u_n|^p\,\text{d}x\to0,
\]
\[
\int_{\mathbb R^N}u_n^2\log u_n^2\,\text{d}x
\to
\int_{\mathbb R^N}u_0^2\log u_0^2\,\text{d}x .
\]
\end{Lemma}

\begin{proof}
By Lemma~\ref{lem2.5}, up to a subsequence, there exists
\(u_0\in H_V^1(\mathbb R^N)\) such that
\[
u_n\rightharpoonup u_0
\quad\text{in }H_V^1(\mathbb R^N),
\]
and \(u_0\) is a weak solution of \eqref{eq1.1}. In particular,
\[
\int_{\mathbb R^N}u_0^2|\log u_0^2|\,\text{d}x<+\infty .
\]
By the compact embedding induced by \((V)\), we also have
\[
u_n\to u_0
\quad\text{in }L^s(\mathbb R^N),\quad 2\le s<2^*,
\]
and
\[
u_n(x)\to u_0(x)
\quad\text{a.e. in }\mathbb R^N .
\]

Since \(u_0\) is a weak solution and
\[
\int_{\mathbb R^N}u_0^2|\log u_0^2|\,\text{d}x<+\infty,
\]
we may take \(u_0\) itself as a test function by a standard cut-off and density
argument. Hence
\begin{equation}\label{eq:limit-identity}
\int_{\mathbb R^N}\bigl(|\nabla u_0|^2+V(x)u_0^2\bigr)\,\text{d}x
=
\int_{\mathbb R^N}u_0^2\log u_0^2\,\text{d}x .
\end{equation}
Since \(I_{\lambda_n}'(u_n)=0\), testing with \(u_n\) gives
\begin{equation}\label{eq:un-identity}
\lambda_n\int_{\mathbb R^N}|u_n|^p\,\text{d}x
+
\int_{\mathbb R^N}\bigl(|\nabla u_n|^2+V(x)u_n^2\bigr)\,\text{d}x
=
\int_{\mathbb R^N}u_n^2\log u_n^2\,\text{d}x .
\end{equation}

Set
\[
F(t)=\bigl(t^2\log t^2\bigr)_+,
\qquad
G(t)=\bigl(t^2\log t^2\bigr)_- .
\]
For every \(q\in(2,2^*)\), there exists \(C_q>0\) such that
\[
F(t)\le C_q |t|^q
\quad\text{for all }t\in\mathbb R .
\]
Since \(u_n\to u_0\) in \(L^q(\mathbb R^N)\), the standard Nemytskii convergence
for subcritical growth yields
\begin{equation}\label{eq:positive-part-conv}
\int_{\mathbb R^N}F(u_n)\,\text{d}x
\to
\int_{\mathbb R^N}F(u_0)\,\text{d}x .
\end{equation}
On the other hand, by Fatou's lemma,
\begin{equation}\label{eq:negative-part-fatou}
\int_{\mathbb R^N}G(u_0)\,\text{d}x
\le
\liminf_{n\to\infty}
\int_{\mathbb R^N}G(u_n)\,\text{d}x .
\end{equation}

Let
\[
A_n=\int_{\mathbb R^N}\bigl(|\nabla u_n|^2+V(x)u_n^2\bigr)\,\text{d}x,
\qquad
A_0=\int_{\mathbb R^N}\bigl(|\nabla u_0|^2+V(x)u_0^2\bigr)\,\text{d}x .
\]
From \eqref{eq:un-identity}, \eqref{eq:positive-part-conv}, and
\eqref{eq:negative-part-fatou}, we obtain
\[
\begin{aligned}
\limsup_{n\to\infty}
\left(
\lambda_n\int_{\mathbb R^N}|u_n|^p\,\text{d}x
+
A_n
\right)
&=
\limsup_{n\to\infty}
\int_{\mathbb R^N}u_n^2\log u_n^2\,\text{d}x  \\
&\le
\lim_{n\to\infty}
\int_{\mathbb R^N}F(u_n)\,\text{d}x
-
\liminf_{n\to\infty}
\int_{\mathbb R^N}G(u_n)\,\text{d}x  \\
&\le
\int_{\mathbb R^N}F(u_0)\,\text{d}x
-
\int_{\mathbb R^N}G(u_0)\,\text{d}x  \\
&=
\int_{\mathbb R^N}u_0^2\log u_0^2\,\text{d}x
=
A_0,
\end{aligned}
\]
where the last equality follows from \eqref{eq:limit-identity}. Since
\(u_n\rightharpoonup u_0\) in \(H_V^1(\mathbb R^N)\), weak lower semicontinuity gives
\[
A_0\le \liminf_{n\to\infty}A_n.
\]
Together with
\[
\lambda_n\int_{\mathbb R^N}|u_n|^p\,\text{d}x\ge0,
\]
we conclude that
\[
A_n\to A_0
\]
and
\[
\lambda_n\int_{\mathbb R^N}|u_n|^p\,\text{d}x\to0 .
\]
The weak convergence \(u_n\rightharpoonup u_0\) in \(H_V^1(\mathbb R^N)\), combined
with \(A_n\to A_0\), yields
\[
u_n\to u_0
\quad\text{in }H_V^1(\mathbb R^N).
\]

Finally, by \eqref{eq:un-identity},
\[
\int_{\mathbb R^N}u_n^2\log u_n^2\,\text{d}x
=
\lambda_n\int_{\mathbb R^N}|u_n|^p\,\text{d}x
+
A_n .
\]
Passing to the limit and using the two convergences just proved, we obtain
\[
\int_{\mathbb R^N}u_n^2\log u_n^2\,\text{d}x
\to
A_0
=
\int_{\mathbb R^N}u_0^2\log u_0^2\,\text{d}x .
\]
The proof is complete.
\end{proof}

\begin{Lemma}\label{lem3.3}
For any fixed \(\lambda\in(0,1]\), the functional \(I_{\lambda}\) admits a sequence of critical points
\(\{u_{\lambda,j}\}_{j=1}^{\infty}\). Moreover, there exist constants
\(\alpha_j,\beta_j\), independent of \(\lambda\), such that
\[
\alpha_j\le I_{\lambda}(u_{\lambda,j})\le \beta_j,
\qquad
\alpha_j\to+\infty
\quad\text{as }j\to\infty .
\]
\end{Lemma}

\begin{proof}
Recall that
\[
X=L^p(\mathbb R^N)\cap H_V^1(\mathbb R^N),
\qquad
\|u\|_X=|u|_p+\|u\|_{H_V^1}.
\]
Let \(\{\phi_j\}\subset C_0^\infty(\mathbb R^N)\) be a linearly independent sequence
which is dense in \(X\). Set
\[
D_n=\operatorname{span}\{\phi_1,\dots,\phi_n\}.
\]
Applying the Gram--Schmidt procedure in \(H_V^1(\mathbb R^N)\), we may assume that
\[
D_n=\operatorname{span}\{e_1,\dots,e_n\},
\]
where \(\{e_j\}\subset C_0^\infty(\mathbb R^N)\) is orthonormal in
\(H_V^1(\mathbb R^N)\). Define
\[
X_n=X\cap D_n^{\perp_{H_V^1}}
=
\left\{
u\in X:\ \langle u,e_i\rangle_{H_V^1}=0,\ i=1,\dots,n
\right\}.
\]
Then \(X_{n+1}\subset X_n\), and
\[
X=D_n\oplus X_n .
\]

For \(\rho>0\), set
\[
\Sigma_\rho=\left\{u\in X:\ \|u\|_{H_V^1}\le \rho\right\}.
\]

We first prove that, for each \(n\in\mathbb N\), there exists \(\rho_n>0\), independent of
\(\lambda\), such that
\[
I_\lambda(u)<0
\quad\text{for all }u\in D_n\setminus\Sigma_{\rho_n}.
\]
Indeed, since all norms are equivalent on \(D_n\), for \(u=sw\), with
\(s>0\), \(w\in D_n\), and \(\|w\|_{H_V^1}=1\), we have
\[
\begin{aligned}
I_\lambda(sw)
&=
\frac{\lambda s^p}{p}\int_{\mathbb R^N}|w|^p\,\text{d}x
+\frac{s^2}{2}\int_{\mathbb R^N}\bigl(|\nabla w|^2+(V(x)+1)w^2\bigr)\,\text{d}x  \\
&\quad
-\frac{s^2\log s^2}{2}\int_{\mathbb R^N}w^2\,\text{d}x
-\frac{s^2}{2}\int_{\mathbb R^N}w^2\log w^2\,\text{d}x  \\
&\le
s^2\left(
C_{1,n}s^{p-2}
-C_{2,n}\log s
+C_{3,n}
\right).
\end{aligned}
\]
Here \(C_{1,n},C_{2,n},C_{3,n}>0\) are independent of \(w\) and \(\lambda\).
Since \(p<2\), it follows that
\[
I_\lambda(sw)\to-\infty
\quad\text{as }s\to+\infty,
\]
uniformly for \(w\in D_n\), \(\|w\|_{H_V^1}=1\), and \(\lambda\in(0,1]\).
Hence such a \(\rho_n\) exists. Moreover, \(\rho_n\) may be chosen as large as needed.

Next, for \(j\in\mathbb N\), define
\[
\theta_j
=
\sup_{\substack{u\in X_{j-1}\\ \|u\|_{H_V^1}=1}}
|u|_{2+\delta},
\qquad
0<\delta<2^*-2 .
\]
We claim that
\[
\theta_j\to0
\quad\text{as }j\to\infty.
\]
Indeed, if not, there exist \(\theta>0\) and \(u_j\in X_{j-1}\) such that
\[
\|u_j\|_{H_V^1}=1,
\qquad
|u_j|_{2+\delta}\ge \theta .
\]
For every fixed \(k\), we have
\[
\langle u_j,e_k\rangle_{H_V^1}=0
\]
for all sufficiently large \(j\). Hence \(u_j\rightharpoonup0\) in
\(H_V^1(\mathbb R^N)\). By the compact embedding induced by \((V)\),
\[
u_j\to0
\quad\text{in }L^{2+\delta}(\mathbb R^N),
\]
which contradicts \(|u_j|_{2+\delta}\ge\theta\). Thus \(\theta_j\to0\).

Set
\[
\rho_j'=\frac{1}{\theta_j}.
\]
Since the radii \(\rho_n\) above can be chosen arbitrarily large, we choose them so that
\[
\rho_n>\max_{1\le k\le n}\rho_k',
\qquad n=1,2,\dots .
\]
Then, for every \(n\ge j\),
\[
\rho_n>\rho_j' .
\]

Define
\[
G_n=
\left\{
H\in C(\Sigma_{\rho_n}\cap D_n,X):
H \text{ is odd and }
H=\operatorname{id}
\text{ on }\partial\Sigma_{\rho_n}\cap D_n
\right\}.
\]
Let \(\Gamma_j\) be the family of all sets of the form
\[
B=H\bigl((\Sigma_{\rho_n}\cap D_n)\setminus Y\bigr),
\]
where \(n\ge j\), \(H\in G_n\), and \(Y\subset\Sigma_{\rho_n}\cap D_n\) is open,
symmetric, and satisfies
\[
\gamma(\overline Y)\le n-j .
\]
Define
\[
c_{\lambda,j}
=
\inf_{B\in\Gamma_j}\sup_{u\in B}I_\lambda(u).
\]

We first prove the upper bound. Since \(0<\lambda\le1\), we have
\[
I_\lambda(u)\le P(u),
\]
where
\[
P(u)
=
\frac{1}{p}\int_{\mathbb R^N}|u|^p\,\text{d}x
+\frac12\int_{\mathbb R^N}\bigl(|\nabla u|^2+(V(x)+1)u^2\bigr)\,\text{d}x
-\frac12\int_{\mathbb R^N}u^2\log u^2\,\text{d}x .
\]
Since \(\Gamma_j\ne\emptyset\), for instance by taking \(n=j\), \(H=\operatorname{id}\),
and \(Y=\emptyset\), we may define
\[
\beta_j
=
\inf_{B\in\Gamma_j}\sup_{u\in B}P(u)<+\infty .
\]
Then
\[
c_{\lambda,j}\le \beta_j
\]
for every \(\lambda\in(0,1]\).

We next prove a lower bound independent of \(\lambda\). By Lemma~\ref{lem3.4}, for every
\(B\in\Gamma_j\),
\[
B\cap \partial\Sigma_{\rho_j'}\cap X_{j-1}\ne\emptyset .
\]
Hence
\[
c_{\lambda,j}
\ge
\inf_{u\in\partial\Sigma_{\rho_j'}\cap X_{j-1}}I_\lambda(u)
\ge
\inf_{u\in\partial\Sigma_{\rho_j'}\cap X_{j-1}}Q(u),
\]
where
\[
Q(u)
=
\frac12\int_{\mathbb R^N}\bigl(|\nabla u|^2+(V(x)+1)u^2\bigr)\,\text{d}x
-\frac12\int_{\mathbb R^N}u^2\log u^2\,\text{d}x .
\]
Using
\[
(t^2\log t^2)_+\le C|t|^{2+\delta},
\]
we obtain, for \(u\in\partial\Sigma_{\rho_j'}\cap X_{j-1}\),
\[
\begin{aligned}
Q(u)
&\ge
\frac12\|u\|_{H_V^1}^2
-C|u|_{2+\delta}^{2+\delta}  \\
&\ge
\frac12\|u\|_{H_V^1}^2
-C\theta_j^{2+\delta}\|u\|_{H_V^1}^{2+\delta}  \\
&=
\frac{1}{2\theta_j^2}-C .
\end{aligned}
\]
Set
\[
\alpha_j=\frac{1}{2\theta_j^2}-C .
\]
Then
\[
c_{\lambda,j}\ge\alpha_j
\quad\text{and}\quad
\alpha_j\to+\infty .
\]

It remains to show that \(c_{\lambda,j}\) is a critical value of \(I_\lambda\)
for all sufficiently large \(j\). Since \(\alpha_j\to+\infty\), there exists
\(j_0\in\mathbb N\) such that
\[
\alpha_j>0
\qquad\text{for all }j\ge j_0 .
\]
Fix \(j\ge j_0\). Then
\[
c_{\lambda,j}\ge \alpha_j>0.
\]
Assume, by contradiction, that \(c_{\lambda,j}\) is not a critical value of
\(I_\lambda\). Then
\[
K_{c_{\lambda,j}}=\emptyset .
\]
Choose \(\overline\theta>0\) so small that
\[
0<\overline\theta<c_{\lambda,j}.
\]
Since \(I_\lambda\) is even and satisfies the Palais--Smale condition,
Lemma~\ref{lem3.1}, applied to \(J=I_\lambda\), gives
\(\theta\in(0,\overline\theta)\) and an odd deformation
\[
\sigma\in C([0,1]\times X,X)
\]
such that
\[
\sigma(0,u)=u,
\]
\[
\sigma(t,u)=u
\quad\text{if }I_\lambda(u)\notin
[c_{\lambda,j}-\overline\theta,c_{\lambda,j}+\overline\theta],
\]
and
\[
I_\lambda(\sigma(1,u))\le c_{\lambda,j}-\theta
\quad\text{whenever }I_\lambda(u)\le c_{\lambda,j}+\theta .
\]
Since
\[
I_\lambda(u)<0
\quad\text{for all }u\in\partial\Sigma_{\rho_n}\cap D_n,
\]
and
\[
c_{\lambda,j}-\overline\theta>0,
\]
the deformation fixes the boundary:
\[
\sigma(t,u)=u
\quad
\text{for all }u\in\partial\Sigma_{\rho_n}\cap D_n,\ t\in[0,1].
\]

By the definition of \(c_{\lambda,j}\), there exists \(B\in\Gamma_j\) such that
\[
\sup_{u\in B}I_\lambda(u)\le c_{\lambda,j}+\theta .
\]
Write
\[
B=H\bigl((\Sigma_{\rho_n}\cap D_n)\setminus Y\bigr)
\]
for some \(n\ge j\), \(H\in G_n\), and \(Y\subset\Sigma_{\rho_n}\cap D_n\)
open, symmetric, with
\[
\gamma(\overline Y)\le n-j .
\]
Define
\[
\widetilde H(u)=\sigma(1,H(u)),
\qquad
u\in\Sigma_{\rho_n}\cap D_n .
\]
Since both \(\sigma(1,\cdot)\) and \(H\) are odd, \(\widetilde H\) is odd.
Moreover, because the deformation fixes the boundary and
\(H=\operatorname{id}\) on \(\partial\Sigma_{\rho_n}\cap D_n\), we have
\[
\widetilde H(u)=u
\quad\text{for all }u\in\partial\Sigma_{\rho_n}\cap D_n .
\]
Thus \(\widetilde H\in G_n\). Hence
\[
\widetilde B
=
\widetilde H\bigl((\Sigma_{\rho_n}\cap D_n)\setminus Y\bigr)
\in\Gamma_j .
\]
For every \(v\in\widetilde B\), there exists \(u\in B\) such that
\[
v=\sigma(1,u).
\]
Since \(I_\lambda(u)\le c_{\lambda,j}+\theta\), the deformation gives
\[
I_\lambda(v)=I_\lambda(\sigma(1,u))\le c_{\lambda,j}-\theta .
\]
Consequently,
\[
\sup_{v\in\widetilde B}I_\lambda(v)\le c_{\lambda,j}-\theta,
\]
which contradicts the definition of \(c_{\lambda,j}\). Therefore
\(c_{\lambda,j}\) is a critical value of \(I_\lambda\) for every \(j\ge j_0\).

Thus, for each \(j\ge j_0\), there exists \(u_{\lambda,j}\in X\) such that
\[
I_\lambda'(u_{\lambda,j})=0,
\qquad
I_\lambda(u_{\lambda,j})=c_{\lambda,j}.
\]
Moreover,
\[
\alpha_j\le I_\lambda(u_{\lambda,j})\le\beta_j,
\qquad
\alpha_j\to+\infty .
\]
Relabeling the sequence \(\{u_{\lambda,j}\}_{j\ge j_0}\), together with the
corresponding \(\alpha_j\) and \(\beta_j\), gives the desired sequence of
critical points. The proof is complete.
\end{proof}

\begin{Lemma}\cite[Lemma~2.4]{MR2988727}\label{lem3.4}
For $B\in\Gamma_{j}$, it follows that $B\cap\partial\Sigma_{\rho}\cap X_{j-1}\neq\emptyset$ provided $\rho<\rho_{n}$ for all $n\geq j$.
\end{Lemma}

\begin{proof}[Proof of Theorem \ref{thm1.2}]
By Lemma \ref{lem3.3}, for every \(\lambda\in(0,1]\), the functional
\(I_{\lambda}\) admits a sequence of critical points
\(\{u_{\lambda,j}\}_{j=1}^{\infty}\subset X\) such that
\[
\alpha_j\le I_{\lambda}(u_{\lambda,j})\le \beta_j,
\qquad
\alpha_j\to+\infty
\quad\text{as }j\to\infty,
\]
where \(\alpha_j,\beta_j\) are independent of \(\lambda\). After discarding
finitely many indices and relabeling, we may assume that \(\alpha_j>0\) for all
\(j\).

Let \(\lambda_n\to0^+\). For each fixed \(j\in\mathbb N\), set
\[
u_{n,j}=u_{\lambda_n,j}.
\]
Then
\[
I_{\lambda_n}'(u_{n,j})=0,
\qquad
\alpha_j\le I_{\lambda_n}(u_{n,j})\le \beta_j .
\]
Applying Lemma \ref{lem3.2} to the sequence \(\{u_{n,j}\}_{n}\), we obtain,
up to a subsequence depending on \(j\), a function \(u_j\in H_V^1(\mathbb R^N)\)
such that
\[
u_{n,j}\to u_j
\quad\text{in }H_V^1(\mathbb R^N),
\]
and \(u_j\) is a weak solution of problem \eqref{eq1.1}. Moreover, Lemma \ref{lem3.2} gives
\[
\lambda_n\int_{\mathbb R^N}|u_{n,j}|^p\,\text{d}x\to0
\]
and
\[
\int_{\mathbb R^N}u_{n,j}^2\log u_{n,j}^2\,\text{d}x
\to
\int_{\mathbb R^N}u_j^2\log u_j^2\,\text{d}x .
\]
Since \(u_{n,j}\to u_j\) in \(H_V^1(\mathbb R^N)\), the quadratic part also
converges. Therefore
\[
I_{\lambda_n}(u_{n,j})
=
I(u_{n,j})
+
\frac{\lambda_n}{p}\int_{\mathbb R^N}|u_{n,j}|^p\,\text{d}x
\to
I(u_j).
\]
Consequently,
\[
I(u_j)
=
\lim_{n\to\infty}I_{\lambda_n}(u_{n,j})
\ge \alpha_j>0 .
\]
In particular, \(u_j\ne0\). Since \(\alpha_j\to+\infty\), we have
\[
I(u_j)\to+\infty
\quad\text{as }j\to\infty .
\]
Thus the solutions \(\{u_j\}_{j=1}^{\infty}\) are mutually distinct, and problem
\eqref{eq1.1} admits infinitely many weak solutions. The proof is complete.
\end{proof}

\section*{Acknowledgment}
C. Huang would like to express sincere gratitude to his supervisor Prof. Zhi-Qiang Wang for his valuable advice. C. Huang was supported by China Postdoctoral Science Foundation (2020M682065). Z. Yang is supported by National Natural Science Foundation of China (12301145,12561020,12261107) and Yunnan Fundamental Research Projects (202301AU070144, 202401AU070123).


\begin{thebibliography}{99}

\bibitem{MR4048330}
C. O. Alves and C. Ji, Existence and concentration of positive solutions for a logarithmic Schr\"odinger equation via penalization method, \textit{Calc. Var. Partial Differential Equations} \textbf{59} (2020), no. 1, Paper No. 21, 27 pp.

\bibitem{MR4097474}
C. O. Alves and C. Ji, Multiple positive solutions for a Schr\"odinger logarithmic equation, \textit{Discrete Contin. Dyn. Syst.} \textbf{40} (2020), no. 5, 2671--2685.

\bibitem{MR1349229}
T. Bartsch and Z.-Q. Wang, Existence and multiplicity results for some superlinear elliptic problems on ${\bf R}^N$, \textit{Comm. Partial Differential Equations} \textbf{20} (1995), no. 9--10, 1725--1741.

\bibitem{MR719365}
T. Cazenave, Stable solutions of the logarithmic Schr\"odinger equation, \textit{Nonlinear Anal.} \textbf{7} (1983), no. 10, 1127--1140.

\bibitem{MR3195154}
P. d'Avenia, E. Montefusco and M. Squassina, On the logarithmic Schr\"odinger equation, \textit{Commun. Contemp. Math.} \textbf{16} (2014), no. 2, 1350032, 15 pp.

\bibitem{MR2570526}
P. Guerrero, J. L. L\'opez and J. Nieto, Global $H^1$ solvability of the 3D logarithmic Schr\"odinger equation, \textit{Nonlinear Anal. Real World Appl.} \textbf{11} (2010), no. 1, 79--87.

\bibitem{MR3451965}
C. Ji and A. Szulkin, A logarithmic Schr\"odinger equation with asymptotic conditions on the potential, \textit{J. Math. Anal. Appl.} \textbf{437} (2016), no. 1, 241--254.

\bibitem{MR1892228}
P. D. Lax, \textit{Functional analysis}, Pure and Applied Mathematics, Wiley-Interscience, New York, 2002.

\bibitem{MR1615056}
V. K. Le and K. Schmitt, On boundary value problems for degenerate quasilinear elliptic equations and inequalities, \textit{J. Differential Equations} \textbf{144} (1998), no. 1, 170--218.

\bibitem{MR1817225}
E. H. Lieb and M. Loss, \textit{Analysis}, 2nd ed., Graduate Studies in Mathematics, vol. 14, American Mathematical Society, Providence, RI, 2001.

\bibitem{MR3259554}
J.-Q. Liu, X.-Q. Liu and Z.-Q. Wang, Multiple sign-changing solutions for quasilinear elliptic equations via perturbation method, \textit{Comm. Partial Differential Equations} \textbf{39} (2014), no. 12, 2216--2239.

\bibitem{MR2983045}
X.-Q. Liu, J.-Q. Liu and Z.-Q. Wang, Quasilinear elliptic equations with critical growth via perturbation method, \textit{J. Differential Equations} \textbf{254} (2013), no. 1, 102--124.

\bibitem{MR2988727}
X.-Q. Liu, J.-Q. Liu and Z.-Q. Wang, Quasilinear elliptic equations via perturbation method, \textit{Proc. Amer. Math. Soc.} \textbf{141} (2013), no. 1, 253--263.

\bibitem{MR845785}
P. H. Rabinowitz, \textit{Minimax methods in critical point theory with applications to differential equations}, CBMS Regional Conference Series in Mathematics, vol. 65, American Mathematical Society, Providence, RI, 1986.

\bibitem{MR1013117}
H. L. Royden and P. M. Fitzpatrick, \textit{Real analysis}, 5th ed., Pearson Education Inc., Hoboken, 2023.

\bibitem{MR3951962}
W. Shuai, Multiple solutions for logarithmic Schr\"odinger equations, \textit{Nonlinearity} \textbf{32} (2019), no. 6, 2201--2225.

\bibitem{MR3385171}
M. Squassina and A. Szulkin, Multiple solutions to logarithmic Schr\"odinger equations with periodic potential, \textit{Calc. Var. Partial Differential Equations} \textbf{54} (2015), no. 1, 585--597.

\bibitem{Struwe}
M. Struwe, \textit{Variational methods: Applications to nonlinear partial differential equations and Hamiltonian systems}, Springer, 1990.

\bibitem{MR3613541}
K. Tanaka and C. Zhang, Multi-bump solutions for logarithmic Schr\"odinger equations, \textit{Calc. Var. Partial Differential Equations} \textbf{56} (2017), no. 2, Paper No. 33, 35 pp.

\bibitem{MR3894545}
Z.-Q. Wang and C. Zhang, Convergence from power-law to logarithm-law in nonlinear scalar field equations, \textit{Arch. Ration. Mech. Anal.} \textbf{231} (2019), no. 1, 45--61.

\bibitem{MR1400007}
M. Willem, \textit{Minimax theorems}, Progress in Nonlinear Differential Equations and their Applications, vol. 24, Birkh\"auser Boston, Inc., Boston, MA, 1996.

\bibitem{MR4066104}
C. Zhang and Z.-Q. Wang, Concentration of nodal solutions for logarithmic scalar field equations, \textit{J. Math. Pures Appl. (9)} \textbf{135} (2020), 1--25.

\end{thebibliography}
\end{document}